\documentclass[11pt]{amsart}
\usepackage{amsmath,amsfonts,amssymb,amsthm,amscd,latexsym,euscript,verbatim,
bbold}
\usepackage[all]{xy}

\makeatletter
\def\section{\@startsection{section}{1}%
  \z@{1.1\linespacing\@plus\linespacing}{.8\linespacing}%
  {\normalfont\Large\scshape\centering}}
\makeatother

\theoremstyle{plain}

\newtheorem*{thmAl}{Albert's Fusion Theorem}

\newtheorem*{conj*}{Root Groups Conjecture}
\newtheorem*{thm1.2}{(1.2) Theorem}
\newtheorem*{thm1.3}{(1.3) Theorem}
\newtheorem*{thm1.4}{(1.4) Theorem}
\newtheorem*{prop*}{Proposition}
\newtheorem*{thm*}{Theorem}

\def\eroman{\etype{\roman}}

\newtheorem{prop}{Proposition}[section]

\newtheorem{thm}[prop]{Theorem}
\newtheorem{cor}[prop]{Corollary}
\newtheorem{lemma}[prop]{Lemma}

\theoremstyle{definition}
\newtheorem{Def}[prop]{Definition}
\newtheorem*{Def*}{Definition}
\newtheorem{Defs}[prop]{Definitions}

\newtheorem{examples}[prop]{Examples}

\newtheorem*{notation*}{Notation}
\newtheorem{remark}[prop]{Remark}

\newcommand{\etype}[1]{\renewcommand{\labelenumi}{(#1{enumi})}}

\newcommand{\la}{\lambda}

\newcommand{\ff}{F}

\newcommand{\zz}{\mathbb{Z}}

\newcommand{\frakA}{\mathfrak{A}}

\newcommand{\ga}{\alpha}
\newcommand{\gb}{\beta}
\newcommand{\gc}{\gamma}

\newcommand{\gd}{\delta}

\newcommand{\gre}{\epsilon}
\newcommand{\gl}{\lambda}

\newcommand{\gr}{\rho}

\newcommand{\charc}{{\rm char}}

\newcommand{\lan}{\langle}
\newcommand{\ran}{\rangle}

\newcommand{\half}{\textstyle{\frac{1}{2}}}

\newcommand{\e}{\mathbb{1}}

\numberwithin{equation}{section}

\begin{document}
\title[Flexible idempotents in nonassociative algebras]{Flexible idempotents  in nonassociative algebras}
\author[Louis Rowen, Yoav Segev]
{Louis Rowen\qquad Yoav Segev}

\address{Louis Rowen\\
         Department of Mathematics\\
         Bar-Ilan University\\
         Ramat Gan\\
         Israel}
\email{rowen@math.biu.ac.il}
\address{Yoav Segev \\
         Department of Mathematics \\
         Ben-Gurion University \\
         Beer-Sheva 84105 \\
         Israel}
\email{yoavs@math.bgu.ac.il}
\thanks{$^*$The first author was supported by the Israel Science Foundation grant 1623/16}


\begin{abstract}
``Fusion rules'' are laws of multiplication among eigenspaces of an
idempotent.
 We establish fusion rules for flexible idempotents, following Albert in the  power-associative
case.  We define the notion of an axis in the noncommutative setting
and  accumulate information about pairs of axes.
\end{abstract}

\date{\today}
\maketitle
\section{Introduction}
Throughout this paper $A$ is an algebra (not necessarily
associative, not necessarily with a multiplicative unit element
$\e$) over a field $F$ with $\charc(F)\ne 2$ unless stated
otherwise. We often denote multiplication $x\cdot y$ in~$A$ by
juxtaposition: $x y$. $A^{(\circ)}$ is the algebra with the same
vector space structure as~$A$, but  where multiplication is defined
by $x\circ y = \half(xy+yx).$  Throughout we denote
\[
C(A)=\{x\in A\mid xy=yx, \text{ for all }y\in A\}.
\]

\begin{Defs}\label{defs main20}$ $
\begin{enumerate}
\item  The algebra $A$ is {\it flexible} if it
satisfies the identity $(xy)x=x(yx).$

\item $A$ is \textbf{power-associative} if $F[x]$ is associative (and therefore
commutative) for each $x \in A.$
\end{enumerate}
\end{Defs}

Commutative algebras are flexible, since
\begin{equation}\label{flex1}
(xy)x=(yx)x=x(yx).
\end{equation}

 In a flexible algebra we write $xyx$ without
parentheses, since there is no ambiguity. Linearizing the flexible
axiom yields \begin{equation}\label{flexiblelin} (xy)z + (zy)x =
x(yz) + z(yx).\end{equation}

We study eigenvalues of idempotents $a\in A,$ in the spirit of
\cite{A}.  We often assume that $A$ is  flexible, our objective
being to generalize the well-known commutative theory.

In \S\ref{id} we also often assume that $A$ is power-associative,
relying heavily on  Albert~\cite{A}.

\begin{enumerate}\eroman\item Following Albert's
terminology, we define the {\it left and right} multiplication maps
 $L_a(b) := a\cdot b $ and
  $R_a(b) := b\cdot a.$

 \item  We write $A_{\la}(X_a)$ for the eigenspace
 of $\la$ with respect to the transformation $X_a$, $X\in \{L,R\},$ i.e., $A_{\la}(L_a) = \{ v \in A: a\cdot v = \la v\},$
and similarly for $A_{\la}(R_a).$ Often we just write $A_\la$ for
$A_{\gl}(L_a)$, when $a$ is understood.

\item
We denote: $A_{\gl,\gd}(a):= A_{\gl}(L_a)\cap A_{\gd}(R_a).$ An
element in $A_{\gl,\gd }(a)$ will be called a $(\gl,\gd)$-{\it
eigenvector} of $a,$  and $(\gl,\delta)$ will be called its {\it
eigenvalue}.
\end{enumerate}

We just write $A_{\gl,\gd }$ when the idempotent $a$ is  understood.
Thus
\[
 A_{\la,\la}=\{ v \in A: a\cdot v = \la v=v\cdot a\},
\]  cf.~\cite[\S 1.1, p.~263]{HSS},
and $a \in A_{1,1}$. Of course $A_{\la,\la}=A_\la $ when $A$ is
commutative. Lemma~\ref{acationscom}  gives identities in the
operators $L_a, R_a\colon A\to A.$
\smallskip

We study idempotents in terms of the following notions.

\begin{Defs}\label{defs main}$ $
\begin{enumerate}

\item  Idempotents $a,b\in A$  are
{\it orthogonal} if $  ab =ba =0.$

\item  An idempotent $a \in A$ is
{\it primitive} if it cannot be written as the sum of nontrivial
orthogonal idempotents.

 \item  An idempotent $a \in A$ is  left (resp.~right)
{\it absolutely primitive} if $A_1(L_a)=Fa$ (resp.~$A_1(R_a)=Fa$).
The idempotent $a$ is {\it absolutely primitive} if it is both left
and right absolutely primitive.

(Sometimes this is called ``primitive'' in the literature.)

 \item
An idempotent $a\in A$  is  {\it flexible} if  $L_aR_a=R_aL_a$, 
and $(xa)x = x(ax)$ for all $x \in A$.
\end{enumerate}
\end{Defs}

\noindent
{\bf Lemma A} (Lemmas~\ref{deg7} and \ref{acationscom}).
Suppose $a\in A$ is a given flexible idempotent.
\begin{enumerate}\eroman

 \item  $R_a^2   - R_a  =  L_a^2 -    L_a.$

\item  $R_a(R_a+L_a-1)=L_a(R_a+L_a-1).$

\item
If  $A$ is power-associative over a field of characteristic $\ne
2,3,$ then
\[
(X_a - 1)  Y_a (L_a + R_a -1) = 0,\text{ for }X,Y\in \{R,L\}.
\]
\end{enumerate}

Next, write
\begin{equation*}\label{eigenspaces}
\mathring{A}_{1/2}(a) =\{x\in A\mid ax+xa=x\}.
\end{equation*}
(We write $ \mathring{A}_{1/2},$ when $a$ is understood.) Albert
proved for $A$ power-associative  with $\charc(F)\ne 2,3$ that
\[
A=A_{1,1} \oplus A_{0,0}\oplus \mathring{A}_{1/2}.
\]
This is reproved   directly as Theorem~\ref{axi}(i).
The following theorem of Albert then describes the appropriate
fusion laws in the algebra~$A:$
\medskip

\noindent
%
\begin{thmAl}[{\cite[Theorem 3, p.~560, Theorem 5, p.~562]{A}}]
Suppose that $A$ is power-associative with $\charc(F)\ne 2,3.$ Let
$a\in A$ be an idempotent.

\begin{enumerate}\eroman
\item
$A_{0,0}, A_{1,1}$ are subalgebras.

\noindent Furthermore for $A$ flexible,
\item  $A_{\la,\gl} \mathring{A}_{1/2} \subseteq A_{1 -\la,1 -\la}+ \mathring{A}_{1/2}, \quad \mathring{A}_{1/2} A_{\la,\la}  \subseteq A_{1 -\la,1 -\la}+ \mathring{A}_{1/2} ,\quad$ for
$\la = 0,1.$
\item
$a \mathring{A}_{1/2} ,  \mathring{A}_{1/2} a\subseteq
\mathring{A}_{1/2}.$
\end{enumerate}
\end{thmAl}

Proposition~\ref{prop 3.4} provides a useful decomposition result into
eigenspaces.

\medskip

In \S\ref{nc1} we drop power-associativity, and turn to a
noncommutative version of axes, in preparation for the study of
flexible axial algebras in \cite{RoSe3}. The simplest nontrivial
situation is for $A_1(L_a)= \ff a$ and for $L_a$ to have exactly one
eigenvalue $\notin \{ 0,1\}$. (When the   algebra $A$ is not
power-associative,   the third eigenvalue need not be $\half$. This
makes \cite{T} quite surprising, giving a condition for $\frac
12$~to be an eigenvalue of an element of $A$ after all.)

Accordingly, we have:

\begin{Def}\label{defs main3}$ $
\smallskip
Let $a\in A$ be an idempotent, and $\gl,\gd\notin\{0,1\}$ in $\ff$.
\begin{enumerate}
\item
 $a$ is a {\it left axis of type $\gl$} if
\begin{itemize}
\item[(a)]
$a$ is  left absolutely primitive, i.e., $A_1 = \ff a$.

\item[(b)]
$(L_a-\gl)(L_a-1)L_a =0.$

\item[(c)]
(cf.~Proposition \ref{prop 3.4}) There is a direct sum decomposition of
$A$ which is a $\zz_2$-grading:
\[
A=\overbrace{A_{0}\oplus A_1}^{\text {$+$-part}}\oplus
\overbrace{A_{\gl}}^{\text{$-$-part}},
\]
recalling that $A_{\gl}$ means $A_{\gl}(L_a)$. We call this grading
the {\it left axial fusion rules}.
\end{itemize}

\item
A {\it right axis of type $\gl$} is defined similarly.
\item
$a$ is an {\it axis (2-sided) of type $(\gl,\gd)$} if $a$ is a left
axis of  type $\gl$ and a    right axis of type $\gd$ and, in
addition,
 $L_a R_a =  R_a  L_a.$
Thus $A_{1,1} =\ff a =A_1$; in particular $A_{1,\gd} = A_{\gl,1} =
0$. Hence
$$A = A^{++} \oplus A^{+-} \oplus  A^{-+} \oplus A^{--},$$
is $\zz_2\times\zz_2$ grading of $A$ (multiplication $(\gre,\gre')(\gr,\gr')$
is defined in the obvious way for $\gre,\gre',\gr,\gr'\in\{+,-\}$),
 where
 \begin{itemize}
\item$ A^{++}=A_{1} \oplus  A_{0,0}, $ \item $A^{+-}= A_{0,\gd},$
 \item$ A^{-+}=  A_{\gl,0},$
 \item $  A^{--}=  A_{\gl,\gd }.$
 \end{itemize}
\item The axis $a$  is of {\it Jordan type $(\gl,\gd)$} if
 $A_{\gl,0} = A_{0,\gl}=0.$

 In this case, $A^{+-}= A^{-+} = 0,$ $A_{0,0} = A_0,$  and
 $$A = A_{1} \oplus A_{0} \oplus A_{\gl,\gd }.$$
\end{enumerate}
\end{Def}
(This definition is close to commutative, but there are natural
noncommutative examples given in Examples~\ref{notJJ} and
\ref{flex}.)

Note that besides being noncommutative,   Definition \ref{defs
main3}(1c) {\bf does not} require the condition that $A_{0}$ is a
subalgebra, contrary to the usual hypothesis in the commutative
theory of axial algebras. But this condition does not seem to
pertain to any of the proofs.

If $a$ is an axis of type $(\gl,\gd),$ then by Definition~\ref{defs
main3}(3),  we can write $x\in A$ as
\[\label{2de}
x=\ga_x a+x_{0,0}+x_{0,\gd}+x_{\gl,0}+x_{\gl,\gd},\qquad
x_{\gc,\gr}\in A_{\gc,\gr},
\]
which we call  the {\it (2-sided) decomposition of} $x$ with respect
to $a.$

When $a$ is of Jordan type, we have the much simpler decomposition
$$x = \ga_xa +x _0 + x_{\gl,\gd}$$
with $x_0\in A_0,$ and $x_{\gl,\gd}\in A_{\gl,\gd },$ which is both
the left decomposition and right decomposition.

 Much of the axial
theory (cf.~\cite{HRS2}) can be
 generalized to this noncommutative setting, as exemplified in the
 following results
\medskip

\noindent
{\bf Proposition B} (Corollary \ref{Jt}).
\begin{enumerate}\eroman
\item
An axis $a$ of Jordan type $(\gl,\gd)$ is in $C(A),$ iff either
$A_{\gl,\gd }=0,$
or $\gl=\gd.$

\item
Let $a\ne b$ be two axes  with $ab=ba$ and $a$ of type  $(\gl,\gd)$.
Then either $ab=0= b_{\gl,\gd},$ or $\gl = \gd.$ In particular, if
$a$ is of Jordan type, then either $a\in C(A)$ or $ab=0$.
\end{enumerate}
\medskip

\noindent
{\bf Proposition C} (Proposition~\ref{twoax}).
If $A$ is an algebra generated by two  axes $a$
and $b$ of Jordan type (whose types may be distinct!), then
$$A = \ff a + \ff b + \ff ab.$$

For flexible algebras we have
\medskip

\noindent
{\bf Theorem D} (Theorem \ref{flex a,b}).
Suppose $A$ is a flexible algebra. Then
\begin{enumerate}
\item  Any axis
$a$ has some  Jordan type $(\gl,\gd)$, with
 either $\gl=\gd,$ so that $a\in C(A),$ or $\gl+\gd = 1$
($\gl\ne\gd$) and $x^2 = 0$ for all $x \in
 A_{\gl,\gd }.$
\item
 Assume that both $a$ and $b$ are not in $C(A).$
Then either $ab=0$ or the subalgebra generated by $a,b$ is one of the
two examples as in Examples
\ref{flex}.
\end{enumerate}

\section{Basic properties of idempotents}\label{id}

 In the following four lemmas (except in part (ii) of Lemma
\ref{acationscom}), one need not assume characteristic $\ne 2$.

\begin{lemma}\label{deg7}
If $a$ is a
flexible idempotent, then $L_a ^2 -L_a =  R_a ^2 -R_a.$
\end{lemma}
\begin{proof} We apply the definition twice:  \begin{gather*}
a +ax + (xa)a + xax =  (a+xa)(a+x) = ((a+x)a)(a+x) =\\
(a+x)(a(a+x)) =  (a+x)(a+ax) =a+a(ax) + xa + xax,
\end{gather*} so $ax + (xa)a =a(ax) + xa$.
\end{proof}

\begin{lemma}\label{acationscom20}
Suppose  $a\in A$ is an idempotent satisfying
 $ L_a^2 -    L_a =R_a^2   - R_a  .$  (In particular this holds when $a$ is a flexible idempotent, by
 Lemma~\ref{deg7}.)
 If $z$ is a  $(\gl,\delta)$-eigenvector,
then either $\delta =\gl$    or $\delta =1-\gl.$
\end{lemma}
\begin{proof}
   $(\gl^2 - \gl)z = a(az) - az = (za)a-za= (\delta^2 -
\delta)z$, implying $\gl^2 - \gl=\delta^2 - \delta,$ or $(\gl -
\delta)(\gl + \delta - 1) = 0.$
\end{proof}
\smallskip

\begin{lemma}\label{acationscom}
Suppose  $a\in A$ is a given flexible idempotent.
\begin{enumerate}\eroman
\item  $R_a(R_a+L_a-1)=L_a(R_a+L_a-1).$

\item
If  $A$ is power-associative over a field of characteristic $\ne
2,3,$ then
\[
(X_a - 1)  Y_a (L_a + R_a -1) = 0,\text{ for }X,Y\in \{R,L\}.
\]
\end{enumerate}
\end{lemma}
\begin{proof}

(i) By Lemma~\ref{deg7},
\begin{gather*}
R_a(R_a+L_a-1)=R_aL_a+R_a^2-R_a=R_aL_a+L_a^2-L_a\\
=L_a(R_a+L_a-1).
\end{gather*}

(ii) We show that $(R_a - 1)  R_a (L_a + R_a -1) = 0,$ the rest
follows from (i).

By \cite[Equation 12, p. 556]{A}, and applying Lemma~\ref{deg7},
 \begin{equation}\begin{aligned}  0 & = R_a^2 + L_a R_a+ L_a^2 - R_a^3 -L_a R_a^2 -L_a \\
 & = R_a^2  + L_a R_a + R_a^2 - R_a^3 -L_a R_a^2 -R_a  \\ &= 2
R_a^2 - R_a^3 +L_a (R_a -R_a^2) -R_a ,\end{aligned}\end{equation}

so
\[(R_a - 1)^2 R_a =  R_a^3 - 2 R_a^2 +R_a  = L_a (R_a -R_a^2) =
- L_a R_a  (R_a -1) ,
\]
so
\[
(R_a - 1)  R_a (L_a + R_a -1) = 0.\qedhere
\]
\end{proof}

\begin{thm}\label{id1}$ $
\begin{enumerate}
\item If $a,a'$ are orthogonal idempotents in $A,$
then they are orthogonal in~$A^{(\circ)}$.

\item Any primitive idempotent of $A^{(\circ)}$  is
a primitive idempotent of  $A$.

 \item Suppose that $A$ is flexible, and let $a$ be an absolutely primitive idempotent of
 $A^{(\circ)}.$ Let $d\in A$ such that $a d=d.$ Then
\begin{itemize}
\item[(i)]
$d a=\ga a,$ and $d^2=\ga d,$ for some $\ga\in F.$
\item[(ii)]
If $\ga=0,$ then $d$ is a $(1,0)$-eigenvector of $a,$ so $a\circ d=\half d,$ and if
$\ga\ne 0,$ then $(a-\ga^{-1}d)$ is a $(1,0)$-eigenvector of $a,$ so $a\circ (a-\ga^{-1}d)=\half (a-\ga^{-1}d).$
\end{itemize}
\end{enumerate}
\end{thm}
\begin{proof}
(1)  This is obvious.

(2) If $a$ is a sum $a' + a''$ of orthogonal idempotents in $A$ then
$a$ is a sum $a' + a''$ of orthogonal idempotents in $A^{(\circ)}$,
by (1).

(3) (i)\ By Lemma \ref{deg7}, $(da)a-da = a(ad)-ad=0.$ Hence $a
\circ da = \frac 12 (ada + (da)a) = da$, so by hypothesis $da \in
Fa$. Writing $da = \ga a$, we have $d^2 = d(ad) = (da)d = \ga a\cdot
d = \ga d.$

\indent (ii)\  This is an easy computation.
\end{proof}

Note  that  Cases (i) and (ii) can arise in matrix algebras,
respectively, with $a
 = e_{11}$ and $d = e_{12},$ or  $a
 = e_{11}+e_{12} $ and $d = e_{11}.$

\begin{remark}
Any algebra $B$ generated by two elements $a, v$ such that $a^2=a,\
av=v,\ va=\ga a,$ and $v^2=\ga v$ with $\ga\in F,$ is flexible and
satisfies $A_1(L_a)=B$ and $A_1(R_a)=Fa.$
\end{remark}

\subsection{Semisimple idempotents}\label{nc11}$ $

\begin{Def}\label{defse}
An idempotent $a$ is {\it left semisimple} of degree $t$ if $$\prod
_ {i=1}^t (L_a -\gl _i ) = 0$$ for suitable distinct eigenvalues
$\gl_1, \dots, \gl_t.$

The idempotent is {\it  semisimple} if it is both left  and right
semisimple.
\end{Def}

\begin{prop}\label{prop 2.7}
Suppose $a$ is a left semisimple idempotent   of degree~$t$. For $y
\in A$, write $y = \sum y_j$ where $y_j$ is the left
$\gl_j$-eigenvector in the left eigenvector decomposition of $y$,
and define the vector space $V_a(y) = \sum _{i=0}^{t-1} \ff\L_a^iy.$
Then
\[V_a(y) =
\oplus_{j=1}^t \ff y_j.
\]
\end{prop}
\begin{proof}
By induction, $L_a^k (y) = \sum \gl_j^k y_j$ for $0 \le k \le t-1.$
Thus $$V_a(y) \subseteq \oplus_{j=1}^t \ff y_j.$$ But on the other
hand, we have a system of $t$ equations in the $y_j$ with
coefficients $( \gl_j^k)$, the Vandermonde matrix, which is
nonsingular, so there is a unique solution for the $y_j$ in $V_a(y)$.
\end{proof}

This space $V_a(b)$ plays a key  role for an idempotent $b$. Here is
another useful general result.

\begin{prop}\label{axi2}  If $a$  is a semisimple  idempotent of a flexible algebra
$A$, then $A$ decomposes as $\oplus _{\gl} A_{\gl,\gl} \oplus
A_{\gl,1-\gl},$ summed over all left eigenvalues $\gl$ of~$a$.
\end{prop}
\begin{proof} These are the only possibilities, in view of
Lemma~\ref{acationscom20}.
\end{proof}

\subsection{Idempotents in power-associative algebras}$ $

Albert  \cite{A} proves some amazing results. \cite[Equation (20),
p.~568]{A} implies that if $\charc(F)\ne 2,3,$ then every idempotent
in a commutative power-associative algebra is an axis (see the
Introduction and Definition \ref{defs main3}). Furthermore, the
eigenspaces are independent by \cite[Equation (23), p.~569]{A}.
\cite[Theorem~I.2, p.~559]{A} computes the \textbf{Albert fusion
rules} with respect to a given idempotent $a$:
\begin{enumerate}\eroman
 \item
$A_1A_1 \subseteq  A_1$ and $A_0A_0 \subseteq A_0;$
\item $A_0A_1 = A_1A_0 =
\{ 0 \};$
\item $A_0 A_{1/2} \subseteq  A_{1/2} + A_1$;
\item $A_1A_{1/2} \subseteq  A_{1/2} + A_0$;
\item $A_{1/2}^2 \subseteq  A_0 + A_1.$
 \end{enumerate}

 We stress the role here of power-associativity.

 In Theorems~I.3, I.5 (\cite[pp.~560, 562]{A}), Albert  gets the same
results, except for (v), for (noncommutative) power-associative
algebras with $\charc(F)\ne 2,3$, replacing the eigenspaces
$A_{\gl},$ with the $\circ$-eigenspaces:
\[
\mathring{A}_{\gl}:=\{x\in A\mid ax+xa=2\gl x\}.
\]
(These are $\frakA_a(\gl)$ in the notation of Albert, \cite[Equation
(27), p.~560]{A}.) He furthermore shows that
$\mathring{A}_{\gl}=A_{\gl,\gl}=A_{\gl},$ and that these are
subrings, for $\gl\in\{0,1\}.$

 For Jordan algebras
cf.~\cite[Equation (26), p. 559]{A} Albert improves  his fusion
rules:

\begin{enumerate}\eroman
 \item
$A_1A_1 \subseteq  A_1$ and $A_0A_0 \subseteq A_0;$
\item $A_1A_0 =
\{ 0 \};$
\item $(A_0 + A_1)A_{1/2} \subseteq  A_{1/2}$;
\item $A_{1/2}^2 \subseteq  A_0 + A_1,$
 \end{enumerate}
but \cite[Equation (26), p. 559, l.-2]{A} has a counterexample  to
(iii) in a commutative  power-associative algebra. When (iii) does
hold, Albert calls the algebra  $A$  {\it stable}.

Many of the above results can be extended to the noncommutative
situation.

\begin{remark}\label{not06}
As we mentioned, in  an easy argument in the first three lines of
the proof of \cite[Theorem 3, p.~560]{A} Albert proves  that
$\mathring{A}_{\gl}=A_{\gl,\gl}$ for $\gl\in\{0,1\}.$ This is used
 in the following result  of Albert on power associative
rings (\cite[(22), p.~559]{A} and \cite[Thm.~3, p.~562]{A}).
\end{remark}

\begin{thm}[\cite{A}]\label{axi}
Suppose that $A$  is power-associative,
and  that $\charc(F)\ne 2,3.$
Let $a\in A$ be an idempotent. Then
\begin{enumerate}\eroman
\item
$A=A_{1,1}\oplus A_{0,0}\oplus \mathring{A}_{1/2}.$

\item
We have
\begin{gather*}
A_1 = X_a (L_a+R_a -1)A; \qquad A_0 = (X_a-1) (L_a+R_a-1)A; \\
  \mathring{A}_{1/2} = (R_a+L_a)(R_a+L_a-2)A, \quad  X\in\{R,L\}.
    \end{gather*}
\end{enumerate}
\end{thm}
Here is an alternate proof that is rather conceptual, working
directly in~$A.$
\begin{proof}
(i)\ Assume first that $A$ is commutative.  Then $A$ is flexible (see equation \eqref{flex1}),
so we may use Lemma~\ref{acationscom}.

Let $x_1=R_ax$ and $x_2=(1-R_a)x,$ so that $x=x_1+x_2.$ By
Lemma~\ref{acationscom}(ii) we have
\begin{equation}\label{eq 3.10(1)}
(R_a-1)(R_a+L_a-1)(x_1)=0=(L_a-1)(R_a+L_a-1)(x_1).
\end{equation}
and
\begin{equation}\label{eq 3.10(2)}
R_a(R_a+L_a-1)(x_2)=0=L_a(R_a+L_a-1)(x_2).
\end{equation}
Let
\begin{alignat*}{3}
x_1'&=(R_a+L_a-1)(x_1), &\qquad\qquad & x_1''=(R_a+L_a-2)(x_1),\\
x_2'&=(R_a+L_a-1)(x_2), &\qquad\qquad & x_2''=(R_a+L_a)(x_2),
\end{alignat*}
Of course
\[
x=x_1'-x_1''+x_2''-x_2'.
\]
By Equations \eqref{eq 3.10(1)} and \eqref{eq 3.10(2)},
\[
R_a(x_1')=x_1'=L_a(x_1'),\qquad R_a(x_2')=0=L_a(x_2').
\]
So
\[
x_1'\in A_1,\qquad\text{and}\qquad x_2'\in A_0.
\]
Also by Equation \eqref{eq 3.10(1)}, $(R_a+L_a-1)(x_1'')=0,$ that is
$ x_1''\in \mathring{A}_{1/2}, $ and by Equation \eqref{eq 3.10(2)},
$(R_a+L_a-1)(x_2'')=0,$ that is $ x_2''\in \mathring{A}_{1/2}.$

For the general case where $A$ is not commutative, $\mathring{A}$ is
commutative and power associative; hence
$A=\mathring{A}_1\oplus\mathring{A}_0\oplus \mathring{A}_{1/2}.$
However, $\mathring{A}_{\gl}=A_{\gl,\gl},$ for $\gl\in\{0,1\},$ by
Remark~\ref{not06}.
\medskip

\noindent (ii)\ follows from (i). Indeed
\[
X_a (L_a+R_a -1)A_0=0=X_a (L_a+R_a -1)\mathring{A}_{1/2},
\]
and $X_a (L_a+R_a -1)(x)=x,$ for $x\in A_1,$ and similarly for
$(X_a-1) (L_a+R_a -1)$ and $(X_a+L_a)(R_a+L_a-2).$
\end{proof}

Note that if $A$ is commutative then $\mathring{A}_{1/2} =A_{1/2} ,$ so
Theorem~\ref{axi} generalizes the commutative description of
idempotents.
\subsubsection{Application to noncommutative Jordan algebras}$ $
\medskip

\begin{Def}(\cite{Sch})\label{defs main2}
    An algebra $A$ is   {\it noncommutative Jordan} if
$A$ is flexible and satisfies the identity
\begin{equation}\label{noncom}
(x ^2w)x = x^2(wx).
\end{equation}
\end{Def}

Schafer~\cite[p.~473]{Sch} showed that any  noncommutative Jordan
algebra with $\e$ of characteristic $\ne 2$  satisfying
\eqref{noncom} is flexible and power-associative.
 Hence, Albert's Fusion Theorem applies to
noncommutative Jordan algebras with $\e$.

\section{Basic properties of axes }\label{nc1}
In this section we generalize certain notions from the theory of
commutative axial algebras in \cite{HRS2}, to the noncommutative
setting.

We start with some examples of axes that are not flexible.

\begin{examples}\label{notJJ}$ $
\begin{enumerate} \eroman
\item
Let $A = \ff a + \ff b +\ff c$ with multiplication given by
\[
 \begin{aligned} &
a^2 = a,  \qquad b^2 = a = c^2,\\ & ab = \gl b, \qquad ba = 0 ,
\\ & ac = 0, \qquad ca = \gd c, \qquad bc =   cb = 0.
\end{aligned}
 \]
 Then
 $a(ba) = 0 = (ab)a$, and   $a(ca) =  0= (ac)a,$ so $L_a R_a = R_a L_a.$ $$(L_a - \gl )b =
 \gl
 (b -b) =0 = L_a c ,$$ implying $(L_a-\gl)(L_a-1)L_a =0.$
Likewise for $R_a.$

Hence $a$  is an axis of  type $(\gl,\gd)$ which is not Jordan.
It is the only nonzero axis in $A$.

\item
Let $0 \ne \gl,\gd\in \ff,$ with $\gl+\gd = 1.$
Let
\[
A=Fa\oplus Fx\oplus Fy\oplus Fy'\oplus Fz,
\]
with multiplication defined by:
\begin{gather*}
a^2=a,\quad ax=xa=0,\quad ay=\gl y, ya=0,\quad ay'=0, y'a=\gd y',\\
az=\gl z,za=\gd z.
\end{gather*}
\begin{gather*}
x^2=xy=yx=xy'=y'x= xz=zx=0.\\
y^2=-x,\quad yy'=0=y'y,\quad yz=\gl y', zy=0.\\
(y')^2=x,\quad y'z=\gd y, zy'=0.\qquad z^2=0.
\end{gather*}
So
\[
A_{1,1}=Fa,\quad A_{0,0}=Fx,\quad A_{\gl,0}=Fy,\quad A_{0,\gd}=Fy',\quad A_{\gl,\gd}=Fz.
\]
It is easy to check that $(av)a=a(va),$ for $v\in\{a,x,y,y',z\},$
and $a$ is an axis of type $(\gl,\gd).$ Of course $a$ is not of Jordan type.
Also, $a$ is not a flexible axis since $(ya)y=0,$ while $y(ay)=\gl y^2=-\gl x.$

Let $b = a + x+y+y'+z.$ Then
\[
\begin{aligned}  b^2 = & a^2 +y^2  +(y')^2 +ay +y'a + az +za +yz + y'z  \\
&= a - x+x +(\gl + \gd) y  + (\gl + \gd) y' + (\gl + \gd)z= b,
\end{aligned}
\]
(We conjecture that by solving some more equations $b$ could also be
made into an axis.)
Note that $A$ has dimension  5. Since $a,b,ab,ba,$ and $aba$ are
independent, they span $A$.
\end{enumerate}
\end{examples}

\begin{remark}\label{not00}
Write $y= y_0 + \ga_y  a+ y_{\gl},$ for $y _\rho \in A_{\rho}$,
$\rho \in \{0,\gl\}$, and $\ga \in \ff$. In this notation,
$y_0 = y_{0,0}+y_{0,\gd},$ and $y_\gl =
y_{\gl,0}+y_{\gl,\gd}.$ When $a$ has Jordan type $(\gl,\gd)$, then
$y_0 = y_{0,0}$ and $y_\gl = y_{\gl,\gd}.$
\end{remark}

\begin{lemma}\label{lem 3.3}$ $
Suppose  $a$ is an axis of type $(\gl,\gd).$
\begin{enumerate}
\item $a(A_{0} + \ff a)= \ff a .$

\item   $ay=\ga_y  a+\gl y_{\gl},$ implying
$$ y_{\gl} =\frac 1 {\gl}
 (ay -\ga_y  a).$$

\item $
a(ay)=\ga_y (1-\gl) a+\gl ay.$
\item $ y_0 =  y -\frac 1 \gl
 (ay -(\gl +1)\ga_y  a)  .$

 \item  $
(ay)a =\ga_y  a+\gl \gd y_{\gl,\gd},$ implying $$  y_{\gl,\gd}
=\frac 1{\gl \gd}
 ((ay)a -\ga_y  a).$$
\end{enumerate}
\end{lemma}
\begin{proof} (1) If $x_0 \in A_0$ then $a x_0 = 0,$ so $a (A_0+A_1)
= 0+ a \ff a = \ff a^2 =  \ff a.$

(2) Apply $L_a$ to  Remark~\ref{not00}.

(3) $ a(ay)=\ga_y  a+\gl^2 y_{\gl}=\ga_y  a+\gl \textstyle{(ay-\ga_y
a)},$  yielding (3).

 (4) $y_0 = y - y_\gl - \ga_y  a = y -\frac
1 \gl
 (ay -(\gl +1)\ga_y  a).$

 (5) Apply $R_a$ to (2).
\end{proof}

\begin{prop}\label{prop 3.4}$ $
\begin{enumerate}\eroman
\item If $a$ is a   left axis of type $\gl$ then the left decomposition $$x = x_{1} +
x_{0} + x_{\gl}$$ of $x\in A$ with respect to $a,$ is given by:
 \begin{itemize}
\item $x_{1} = \frac 1{\gl -1}(a(a-\gl))x = \ga_x a, \ \ga_x \in \ff,$
\item $x_{0} =  -\frac 1{\gl}((a-\gl)(a-1))x \in A_0(L_a)$,
\item $x_{\gl} =  -\frac 1{\gl(\gl -1)}(a(a-1))x  \in A_{\gl}(L_a)$.
 \end{itemize}
 \item If $a$ is a    right axis of type $\gd$ then we have an analogous further right decomposition $$x_\rho = x_{\rho,1} +
x_{\rho,0} + x_{\rho,\gd}$$ of $x_\rho \in A_\rho,$ (for $\rho\in
\{0,1,\gl\}$ with respect to $a,$ given by:
 \begin{itemize}
\item $x_{\rho,1} = \frac 1{\gd -1}x_\rho(a(a-\gd)) = \ga_{x_\rho} a, \ \ga_{x_\rho} \in \ff,$
\item $x_{\rho,0} =  -\frac 1{\gd}x_\rho((a-\gd)(a-1)) \in A_0(R_a),$
\item $x_{\rho,\gd} =  -\frac 1{\gd(\gd -1)}x_\rho(a(a-1))  \in A_{\gd}(R_a).$
 \end{itemize}

\item  An idempotent $b$ cannot be a left eigenvector of a left axis $a\ne b$ unless
$ab =0.$

\item
If $a$ is an axis of  type $(\gl,\gd),$ then $A$ decomposes into a
direct sum
\[
A=\overbrace{A_{1,1} \oplus A_{0,0}}^{\text {$++$-part}}\oplus
\overbrace{A_{0,\gd}}^{\text{$+-$-part}} \oplus \overbrace{
A_{\la,0}}^{\text {$-+$-part}}\oplus \overbrace{A_{\gl,\gd
}}^{\text{$--$-part}},
\]
and this is a $\mathbb Z_2\times \mathbb Z_2$-grading of $A.$ In
particular $ A_{0,1}= A_{\gl,1}= A_{1,0}= A_{1,\gd}= 0.$

\item
If $a$ is an axis of Jordan type $(\gl,\gd)$ then
\[
A=\overbrace{A_{1,1} \oplus A_{0,0}}^{\text {$+$-part}}\oplus
\overbrace{A_{\gl,\gd }}^{\text{$-$-part}},
\]
is a $\zz_2$-grading of $A.$
\end{enumerate}
\end{prop}
\begin{proof}
Note that
\begin{equation}\label{laxis}
\textstyle{
\frac 1{\gl -1}(L_a(L_a-\gl)) -\frac 1{\gl}((L_a-\gl)(L_a-1)) -\frac
1{\gl(\gl -1)}(L_a(L_a-1)) = 1,}
\end{equation}
so their sum applied
to $x$ is $x$, and each term is in the appropriate left eigenspace.

(ii) As in (i), using equation \eqref{laxis}, with $R_a$ in place of
$L_a,$ noting that $L_a$ and $R_a$ commute.

(iii) If the idempotent $b\ne a$  is a left eigenvector,  then,
since $a$ is left absolutely primitive, its eigenvalue must  be $0$
or $\gl$. But if $b\in A_{\gl}(L_a),$ then, by Definition \ref{defs
main3}(1c), $b^2 \in A_0(L_a)+A_1(L_a),$ but $b^2=b,$ a
contradiction.

 (iv) Combining (i) and (ii) together
with the left $\zz_2$ and the right $\zz_2$ grading given in
Definition \ref{defs main3}((1)\& (2)), shows that:
\[
A=\overbrace{A_{1,1} \oplus A_{1,0} \oplus A_{0,1}\oplus
A_{0,0}}^{\text {$++$-part}}\oplus \overbrace{A_{1,\gd} \oplus
A_{0,\gd}}^{\text{$+-$-part}} \oplus \overbrace{A_{\la,1}\oplus
A_{\la,0}}^{\text {$-+$-part}}\oplus
\overbrace{A_{\gl,\gd}}^{\text{$--$-part}},
\]
is a $\zz_2\times\zz_2$ grading of $A.$

We have $\ga_x a = x_{1,1} + x_{1,0}+ x_{1,\gd},$ so $x_{1,1} = \ga_x a$ 
and $ x_{1,0}= x_{1,\gd}= 0.$

Working from the other side yields $ x_{0,1}= x_{\gl,1}= 0.$

(v) The other components vanish, by the definition of Jordan type.
\end{proof}

\begin{cor}\label{Jt}$ $

\begin{enumerate}\eroman\item An axis  $a$ is in the center of $A$, iff it is of Jordan type
with $A_{\gl,\gd }=0.$

\item An axis $a$ of Jordan type $(\gl,\gd)$ is in $C(A),$ iff either
$A_{\gl,\gd }=0$ or $\gl=\gd.$

\item
Let $a\ne b$ be two axes  with $ab=ba$ and $a$ of type  $(\gl,\gd)$.
Then either $ab=0,$ or $\gl = \gd.$ In particular, if $a$ is of
Jordan type, then either $a\in C(A),$ or $ab=0$.
\end{enumerate}
\end{cor}
\begin{proof}
(i) $(\Rightarrow)$ For any $x\in A$,
\[
\gl^2 x_\gl + \ga_x a= a(ax) = ax = \gl x_\gl + \ga_x a,
\]
implying $x_\gl = 0$ since $\gl \ne 0,1.$ Hence $x_{\gl,0}=0 =
x_{\gl,\gd},$ and by symmetry $x_{0,\gd}=0.$

$(\Leftarrow)$ $x = \ga _x a+x_{0,0}$ for any $x$ in $A$, implying
$ax = xa = \ga_x a$ and also $xy = x_{0,0}y_{0,0} + \ga _x \ga _y
a,$ and thus $a(xy) = (ax)y = \ga _x a y = \ga _x \ga_y a$ for all
$x,y$.

 (ii) Suppose that $a\in C(A)$  and that $A_{\gl,\gd }\ne 0.$
Let $0\ne x_{\gl,\gd}\in A_{\gl,\gd }.$ Then $\gl x=ax=xa=\gd x,$ so
$\gl=\gd.$

Conversely, for $x\in A,$   decompose $x$ according to $a:$ $x  =
\ga a+x_{0,0}+ x_{\gl,\gl}$, where $\ga \in \ff$ and $x_{\gl,\gl}$
may be $0.$ Then $ax=\ga a+\gl x_{\gl,\gl}=xa,$ so $a\in C(A).$

(iii) Let $a$  be of type $(\gl,\gd)$  and let $b=\ga_b a+b_{0,0}
+b_{\gl,0} +b_{0,\gd}+b_{\gl,\gd}$ be the decomposition of  $b$ with
respect to~$a$. Then $\ga_b a +\gl b_{\gl,0}+\gl b_{\gl,\gd}
=ab=ba=\ga_b a +\gd b_{0,\gd}+\gd b_{\gl,\gd}$ implying
$b_{\gl,0}=b_{0,\gd}=0.$ If $b_{\gl,\gd}\ne 0,$ then $\gl =\gd$, and
we are done. Hence $ b_{\gl,\gd}=0$, in which case $ab \in \ff a.$
But then $a$ is an eigenvector of $b$, so by Proposition
\ref{prop 3.4}(iii), $ab =0.$
\end{proof}

\subsection{Axes in flexible algebras} $ $

Next we analyze the situation in flexible algebras. Here are two
examples of  flexible algebras that are not commutative.

\begin{examples}\label{flex}
Let $\gl,\gd\in\ff$ such that $\gl,\gd\notin\{0,1\},$
$\gl+\gd=1,$ and $\gl\ne\gd.$
\medskip

\noindent (1)\ Let $A$ be  the $2$-dimensional algebra $\ff a+\ff b$
with multiplication defined by
\[
a^2=a,\quad  b^2=b,\quad ab=\gd a+\gl b\quad ba=\gl a+\gd b.
\]
Then $(ab)a=\gd a+\gl ba=\gd a+\gl^2a+\gl\gd b$
and $a(ba)=\gl a+\gd ab=\gl a+\gd^2a+\gl\gd b.$
Note however that $\gd+\gl^2=\gl+\gd^2.$ So $(ab)a=(ab)a.$
By symmetry $(ba)b=b(ab).$  Next we show that $R_a^2-R_a=L_a^2-L_a$.
We have
\begin{gather*}
a(ab)-ab=a(\gd a+\gl b)-\gd a-\gl b=\gl ab-\gl b=\gl(\gd a+\gl b)-\gl b\\
=\gl\gd a+\gl^2b-\gl b.
\end{gather*}
and
\begin{gather*}
(ba)a-ba=(\gl a+\gd b)a-\gl a-\gd b=\gd ba-\gd b=\gd(\gl a+\gd b)-\gd b\\
=\gd\gl a+\gd^2b-\gd b.
\end{gather*}
Since $\gl^2-\gl=\gd^2-\gd,$ we are done.

Next
\begin{gather*}
((\ga a+\gb b)a)(\ga a+\gb b)=(\ga a+\gb ba)(\ga a+\gb b)=\\
\ga^2 a+\ga\gb ab+\gb\ga (ba)a+\gb^2(ba)b,
\end{gather*}
and
\begin{gather*}
((\ga a+\gb b)(a(\ga a+\gb b))=(\ga a+\gb b)(\ga a+\gb ab)=\\
\ga^2 a+\ga\gb a(ab)+\gb\ga ba+\gb^2b(ab).
\end{gather*}
Since $(ba)a-ba=a(ab)-ab,$ and since $(ab)a=a(ba),$ we see that
$(xa)x=x(ax),$ for all $x\in A.$  By symmetry $(xb)x=x(bx),$ for all $x\in A,$
so $A$ is flexible.

We have
\begin{align*}
&a(a-b)=a-ab=a-(\gd a+\gl b)=\gl(a-b)\text{ and }\\
&b(a-b)=ba-b=\gl a+\gd b-b=\gl(a-b).
\end{align*}
Similarly $(a-b)a=\gd(a-b)=(a-b)b$. (Indeed $(a-b)^2=0.$) We thus
have $A_{\gl,\gd }(a)=A_{\gl,\gd }(b)=\ff(a-b)$, with
$(\ff(a-b))^2=0,$ and of course $A_0(a)=A_0(b)=0.$  Thus the fusion
laws hold for both $a$ and $b$, and both are  axes of Jordan type
$(\gl,\gd)$. Note that the idempotents in $A$ have the form $\ga
a+(1-\ga)b,\ \ga\in\ff.$
\medskip

\noindent (2)\
 Let
$A =\ff a+\ff b+\ff x,$ with multiplication defined by $a^2=a,$ \
$b^2=b,\ x^2=0,$ and
\[
ab = ax =xb =\gl x,\qquad ba=xa=bx=\gd x.
\]
 So
\[
\ff x = A_{\gl,\gd },\qquad A_{0,0}=\ff (b-x).
\]

In checking the fusion rules for $a$, we have $(b-x)x = bx \in \ff
x$ and $x^2 = 0,$ and also
$$(b-x)^2 = b^2 -bx -xb + x^2 = b -(\gl + \gd)x =b-x.$$

Likewise for $b$. Thus $a$ and $b$ are axes of
 Jordan type $(\gl,\gd)$ and $(\gd,\gl)$ respectively.

Let us check the eigenvalues of the idempotent $b-x.$ $a(b-x) = 0 =
(b-x)a.$ $(b-x)x = \gd x$, and $x(b-x) = \gl x$. Thus $b-x$ also is
a Jordan axis, of type $(\gd,\gl).$

Note that the idempotents in $A$ all  have the form $\ga a+\gb b+\xi
x,$ with $\ga+\gb=1.$

To show that $A$ is flexible, we compute that
\begin{equation*}
\begin{aligned}
&((\ga a+\gb b +\xi x)a)(\ga a+\gb b +\xi x)\\
&=(\ga a+\xi x)(\ga a+\gb b +\xi x)\qquad \eta=\gb\gd+\xi\gd\\
&=\ga^2 a+\ga\gb\gl x+\ga\xi\gl x+ \eta(\ga\gd+\gb\gl)x\\
&=\ga^2 a+\ga\gb\gl x+\ga\xi\gl x+(\ga\gb\gd^2+\gb^2\gl\gd+\ga\xi\gd^2+\xi\gd\gb\gl)x
\end{aligned}
\end{equation*}
and
\begin{equation*}
\begin{aligned}
&(\ga a+\gb b +\xi x)(a(\ga a+\gb b +\xi x))\\
&= (\ga a+\gb b +\xi x)(\ga a+\gc x)\qquad \gc=\gb\gl+\xi\gl\\
&=\ga^2 a+\ga\gb\gd x+\ga\xi\gd x+ \gc(\ga\gl+\gb\gd)x\\
&=\ga^2 a+\ga\gb\gd x+\ga\xi\gd x+(\ga\gb\gl^2+\gb^2\gl\gd+\xi\ga\gl^2+\xi\gl\gb\gd)x
\end{aligned}
\end{equation*}
We must show that
\[
\ga\gb\gl+\ga\xi\gl+\ga\gb\gd^2+\ga\xi\gd^2=\ga\gb\gd+\ga\xi\gd +\ga\gb\gl^2+\xi\ga\gl^2.
\]
But this follows from the fact that $\gl^2-\gl=\gd^2-\gd.$ By
symmetry $(yb)y=y(by),$ for all $y\in A.$ It is easy to check that
$(yx)y=y(xy),$ for all $y\in A,$ and we see that $A$ is flexible.
\end{examples}

\begin{thm}\label{flex a,b}
Suppose $A$ is a flexible algebra. Then
\begin{enumerate}
\item Any axis
$a$ has some Jordan type $(\gl,\gd)$ and is either in $C(A)$ (with
$\gl=\gd)$, or  $\gd = 1-\gl,$ $\gl\ne\half,$ and $x^2 = 0$ for all
$x \in
 A_{\gl,\gd }.$
\item For axes $a, b\in A$,
 write
\[
b=\ga_b a+b_0+b_{\gl,\gd}\text{ and }a=\ga_a b+a_0+a_{\gl',\gd'},
\]
where $b_0\in A_0(a),\ b_{\gl,\gd}\in A_{\gl,\gd }(a),\ a_0\in
A_0(b),\ a_{\gl',\gd'}\in A_{\gl',\gd'}(b).$  Assume that
$\gl\ne\gd,$ and $\gl'\ne\gd'$ (so that both $a$ and $b$ are not in
$C(A)$, and $\gd = 1- \gl$ and $\gd' = 1-\gl'$).
\begin{itemize}
\item[(I)]
$(ab)^2=\ga_b (ab),$ and $(ba)^2=\ga_b (ba).$
In particular, if $ab\ne 0,$ then $\ga_a=\ga_b.$

\item[(II)]
Suppose that $b_{\gl,\gd}\ne 0,$ so that $ab\ne 0,$ and hence
$\ga_a=\ga_b.$ Write $\ga:=\ga_b,$ and let $\lan a,b\ran$ be the
subalgebra generated by $a,b.$ Then
\begin{itemize}
\item[(i)]
$b_0b_{\gl,\gd}=\gd(1-\ga)b_{\gl,\gd},$ and $b_{\gl,\gd}b_0=\gl(1-\ga)b_{\gl,\gd}.$

\item[(ii)]
$bb_{\gl,\gd}=(\ga(\gl-\gd)+\gd)b_{\gl,\gd}$ and $b_{\gl,\gd}b=(\ga(\gd-\gl)+\gl)b_{\gl,\gd}.$
In particular $a_{\gl',\gd'}$ is a scalar multiple of $b_{\gl,\gd}.$

\item[(iii)]
$\lan a, b\ran$ is spanned by $a,b, b_{\gl,\gd}.$

\item[(iv)]
If $\ga\ne 0,$ then $\ga =1,$ and  $ \lan a, b\ran=\ff a+\ff b$ is
isomorphic to the algebra in Example \ref{flex}(i). Hence $\gl'
=\gl.$

\item[(v)]
If $\ga=0,$ then $\lan a, b\ran$ is isomorphic to the algebra in
Example \ref{flex}(ii).  Hence $\gl' =1-\gl.$

\end{itemize}
\item[(III)]
Either $ab=ba=0$ or $\lan a, b\ran$ is isomorphic to one of the
algebras in Examples \ref{flex}.
\end{itemize}
\end{enumerate}
\end{thm}

\begin{proof}
(1) The first assertion is by Proposition~\ref{axi2}. Suppose $a$ is
an  axis of Jordan type $(\gl,\gd)$. Then
$\gd  x^2 =(xa) x =  x(a  x) = \gl x^2,$ for $x\in A_{\gl,\gd }.$
Hence $\gl = \gd$ unless $ x^2 = 0$ for all $x\in A_{\gl,\gd }$.
\medskip

\noindent
(2)(I)\
We have
\[
ab=\ga_b a+\gl b_{\gl,\gd}\implies b_{\gl,\gd}=\gl^{-1}(ab-\ga_b a).
\]
Hence, since $ab_{\gl,\gd}+b_{\gl,\gd}a=b_{\gl,\gd},$ and $b_{\gl,\gd}^2=0,$ we get
\[
(ab)^2=(\ga_b a+\gl b_{\gl,\gd})^2=\ga_b^2 a+\ga_b (ab-\ga_b a)=\ga_b (ab).
\]
The proof that $(ba)^2=\ga_b (ba)$ is similar.
\medskip

\noindent (2)(II)\ (i)\ We have
\[
\begin{aligned}
b=b^2&=(\ga_ba+b_0+b_{\gl,\gd})^2\\
&=\ga_b^2 a+b_0^2+\ga_b(\gl+\gd)b_{\gl,\gd}+b_0b_{\gl,\gd}+b_{\gl,\gd}b_0.
\end{aligned}
\]
Hence, by the $\zz_2$ grading of $A,$ we have
\begin{equation}\label{b0}
b_0^2=(\ga_b-\ga_b^2)a+b_0,\text{ and}
\end{equation}
\begin{equation}\label{eq 3.2}
(1-\ga_b)b_{\gl,\gd}=b_0b_{\gl,\gd}+b_{\gl,\gd}b_0.
\end{equation}
Furthermore, linearizing the flexible axiom yields
\[
(xy)z + (zy)x = x(yz) +z(yx).
\]
Taking  $b_0$ for $x$ and $a$ for $y$ yields $xy=yx =
0$, implying $(za)b_0 = b_0(az),$ and thus, for $z = b_{\gl, \gd}$
we have
\begin{equation}\label{line}
\gd b_{\gl,\gd} b_0 = \gl b_0 b_{\gl,\gd}.
\end{equation}
Combining equations \eqref{eq 3.2} and \eqref{line} and the fact that
$\gl+\gd=1,$ we get (i).
\medskip

\noindent (ii)  This is an easy computation using (i), $\gl+\gd=1,$
and the fact that $b_{\gl,\gd}^2=0.$
\medskip

\noindent
(iii)  Set $B:=\ff a+\ff b+\ff b_{\gl,\gd}.$ Clearly $ab,\ ab_{\gl,\gd},\ ba,\ b_{\gl,\gd}a\in B.$  By (ii), also $bb_{\gl,\gd}, b_{\gl,\gd}b\in B,$
and $b_{\gl,\gd}^2=0.$
\medskip

\noindent
(iv)\ Suppose $\ga\ne 0.$
By (ii) we have
\[
ab=\ga b+\gc b_{\gl,\gd},\qquad\gc\in\ff.
\]
Hence
\[
0=ab-ab=\ga(a-b)+\gr b_{\gl,\gd},\qquad \gr\in\ff.
\]
This shows that $b_{\gl,\gd}$ is a multiple of $a-b.$  By (iii), $\lan a,b\ran=\ff a+\ff b.$
In particular $\lan a,b\ran_0(b)=0,$ so $b=\ga a+b_{\gl,\gd}.$  Since $b_{\gl,\gd}$ is
a multiple of $a-b,$ we must have $\ga =1.$  It is now easy to check that (iv) holds.
\medskip

\noindent
(v)\
 Suppose that $\ga=0.$
Then, by equation \eqref{b0}, $b_0$ is an idempotent, and by (i),
\[
b_0b_{\gl,\gd}=\gd b_{\gl,\gd}\text{ and }b_{\gl,\gd}b_0=\gl b_{\gl,\gd}.
\]
But now $b_0$ is an idempotent and $b=b_0+b_{\gl,\gd},$ where
$b_{\gl,\gd}\in A_{\gd,\gl}(b_0).$ Also $bb_0\ne b_0b.$ By (iv) the
subalgebra generated by $b$ and $b_0$ is as in Example
\ref{flex}(i). Hence, since $b_0$ is of type $(\gd,\gl)$, $b$ is
also of type $(\gd,\gl)$.

Now $ab=\gl a_{\gd,\gl},$ but, since $\ga=0,$ also $ab=\gl b_{\gl,\gd}.$
Hence $a_{\gd,\gl}=b_{\gl,\gd}.$  We now see that $\lan a,b\ran$
is the algebra as in Example \ref{flex}(ii), with $x=b_{\gl,\gd}.$
\medskip

\noindent (III)
Assume that $ab\ne 0.$ If $b_{\gl,\gd}=a_{\gl',\gd'}=0,$ then
$ab=\ga_ba=ba=\ga_ab.$  this is possible iff $\ga_a=\ga_b=0,$
but then $ab=0.$ Hence we may assume without loss that $b_{\gl,\gd}\ne 0,$
and (III) holds by (II).
\end{proof}
\begin{thm}\label{twoax0}
Suppose $a$ is a   left axis, and $b$ is a right axis.  Then
\begin{enumerate}\eroman

\item  $a(ax)  \in \ff a  + \ff ax,$ for all $x\in A.$

\item $(xb)b\in \ff b + \ff xb,$ for all $x\in A.$

\item
Let
$$V = \ff a + \ff b + \ff ab + \ff ba.$$
Then $(ab)(ab) - a(bab) \in V.$
\item
Suppose that $a$ and $b$ are  axes (whose types may be distinct!),
and  that $A$ is generated by $a$ and $b.$
\begin{enumerate}
\item If $V$ (of
(iii)) contains $aba$,  $bab$, $(ab)(ba),$ and $(ba)(ab) ,$ then
$V=A.$

\item Let
$$V' = \ff a + \ff b + \ff ab + \ff ba +\ff aba +\ff bab.$$ Then
$aV', bV', V'a, V'b \subseteq V'.$
\end{enumerate}
\end{enumerate}
\end{thm}

\begin{proof} (i) , (ii) Special cases of Lemma~\ref{lem 3.3}.
\medskip

\noindent
(iii)  We write $y \cong z$ if $y-z\in V.$ Now $(ab - \gl b)\in
Fa+A_0(L_a),$ so
\[
b_{\gl} (ab - \gl b) \in
 A_{\gl}.
\]
Note that by (ii), $(ab)b\in V.$ Note also that applying $L_a$ to
(i) implies $a(a(ab))\in V.$ Hence
 \begin{equation*}
\begin{aligned}
(ab) (ab) &\cong ab (ab - \gl b)= \gl b_{\gl} (ab - \gl b) + \ga_ b a  (ab - \gl b)\cong \gl b_{\gl} (ab - \gl b)\\
 &= a(b_{\gl} (ab - \gl b))= a((b-\ga_b a - b_{0})(ab - \gl b))\\
& = a((b - \ga_b a)(ab))-  \gl a((b - \ga_b a)b) - a(b_{0} (ab - \gl
b))\\ & = a(bab)  - \ga_b a(a(ab))-  \gl a b + \gl \ga_b a(ab) -
a(b_{0} (ab - \gl b)) \\& \cong a(bab),
\end{aligned}
\end{equation*}
in view of (i), since $ a(b_0 (ab - \gl
 b)) \in a(A_{0}+ \ff a) \subseteq \ff a.$
\medskip

\noindent
 (iv(a)) We check that  $V$ is closed under products.

By (i) and (ii) and since $aba, bab\in V,$ we see that $V$ is closed
under $L_a, R_a, L_b$ and $R_b.$

 By (iii), $(ab)(ab) \cong a(bab)\in aV \subseteq V$.
  Reversing the roles of $a$ and $b$ shows that $(ba)(ba)\in
V$. We are given that  $(ab)(ba)\in V $ and $(ba) (ab)  \in V.$
\medskip

\noindent
  (iv(b)) By symmetry (both with respect to $a,b$ and to working
    on the left or on the right), it is enough to show that $aV'\subseteq V'$.
By (i), $a(aba)= a(a(ba))\in \ff a+\ff aba\subseteq V'.$

It remains to check that $a(bab)\in V'$. By (iii) and (i),
$$a\big((ab)(ab)-a(bab)\big)\in V'.$$ By  Lemma~\ref{lem 3.3}(3),
$a(a(bab))-\gl a(bab)\in V'$   implying $a((ab)(ab))-\gl a(bab)\in
V'.$

Let the type of $a$ be $(\gl,\gd).$ We work with the decomposition
of $b$ with respect to $a.$ Now $aba=\ga_ba+\gl\gd b_{\gl,\gd},$ so
$b_{\gl,\gd}\in V'.$ Hence $b_{\gl,0}\in V',$ because $ ab=\ga_b
a+\gl b_{\gl,0}+\gl b_{\gl,\gd}\in V'.$ by the fusion rules of
Proposition~\ref{prop 3.4}(iv), that $a((ab)(ab))$ is an $\ff$-linear
combination of $a, b_{\gl,0}$ and $b_{\gl,\gd}.$ Hence
$a((ab)(ab))\in V',$ so $a(bab)\in V'.$
\end{proof}

 In \cite[Theorem 1.10]{RoSe3} we shall prove that the conclusion of
(iv(b)) implies $A = \ff a + \ff b + \ff ab + \ff ba +\ff aba +\ff
bab,$ by means of an extra technique of ``Miyamoto involutions.''
This result can be improved in certain situations:

\begin{cor}\label{twoax}
Suppose $A$ is an algebra generated by two  axes $a$ and $b$ (whose
types may be distinct!), and let
$$V = \ff a + \ff b + \ff ab .$$
\begin{enumerate}\eroman
\item If $ba \in V$, then $V=A.$
\item In
particular, (i) is the case when $a$ or $b$ are  axes of Jordan
type.

\end{enumerate}
\end{cor}
\begin{proof} (i)
By Theorem~\ref{twoax0}(i), $V$ is closed under $L_a.$ In particular
$aba\in V,$ so $V$ is closed under $R_a.$ Similarly $V$ is closed
under $R_b$ by Theorem~\ref{twoax0}(ii), so $bab\in V,$ and then $V$
is closed under $L_b.$ It remains to show that $(ab)(ab)\in V,$
which follows from Theorem~\ref{twoax0}(iii).

(ii) (For $a$ of Jordan type) $ba=\ga_b a+\gd b_{\gl,\gd},$ and
$b_{\gl,\gd}\in V$ by Theorem~\ref{twoax0}(i), so $ba\in V.$
\end{proof}

 We cannot improve (ii) directly to arbitrary axes.
 However, we have a positive result,  which will be developed further  in
 \cite{RoSe3}.

\begin{prop}\label{prop 3.10}
Let $a$ be an axis of type $(\gl,\gd)$. For $x \in A$, write $x =
\ga_xa + x_{0,0} + x_{\gl,0} + x_{0,\gd}+ x_{\gl,\gd}$ as in
\eqref{2de}, and define the vector space
\[
 V_a(x) = \ff a\oplus\ff x\oplus\ff ax\oplus \ff xa\oplus \ff axa.
\]
 Then
\[V_a(x) = \ff a
\oplus  \ff x_{0,0} \oplus \ff  x_{\gl,0} \oplus \ff x_{0,\gd}\oplus
\ff x_{\gl,\gd}.
\]
\end{prop}
\begin{proof}
Applying a special case of Proposition~\ref{prop 2.7} both on the left and right,
\[\sum _{i,j=0}^{2} \ff L_a^i
R_a^jx =\ff a \oplus  \ff x_{0,0} \oplus \ff  x_{\gl,0} \oplus \ff
x_{0,\gd}\oplus \ff x_{\gl,\gd}.\]  We conclude with Lemma~\ref{lem 3.3}(3), which lets us replace $L_a^2$ by $L_a$, and $R_a^2$
by~$R_a$ (when we add $\ff a$).
\end{proof}

\end{document}